\documentclass[11pt]{amsart}
\usepackage[parfill]{parskip} 
\usepackage{amsmath, amsxtra, amsthm, amssymb,eucal,dsfont}
\usepackage[all]{xy}
\usepackage{color}
\usepackage[colorlinks=true, pdfstartview=FitV, linkcolor=blue, citecolor=blue, urlcolor=blue]{hyperref}

\usepackage{graphicx}

\numberwithin{equation}{section}
\swapnumbers

\newtheorem{thm}{Theorem}[section]
\newtheorem{prop}[thm]{Proposition}
\newtheorem{lem}[thm]{Lemma}
\newtheorem{rmk}[thm]{Remark}

\newcommand{\nc}{\newcommand}
\nc{\rc}{\renewcommand}

\nc{\se}{\section}
\nc{\sse}{\subsection}

\rc{\b}{\mathbb}
\rc{\c}{\mathcal}
\nc{\tn}{\textnormal}

\nc{\bZ}{\b Z}

\nc{\cF}{\c F}
\nc{\cN}{\c N}
\nc{\cP}{\c P}
\nc{\cQ}{\c Q}

\nc{\al}{{\alpha }}
\nc{\be}{{\beta }}

\nc{\inv}{ ^{-1}}
\nc{\su}{\subset}
\nc{\lan}{\langle}
\nc{\ran}{\rangle}
\nc{\rk}{\tn{rk}}
\nc{\cover}{\;\tilde\hookrightarrow\;}

\title{Symmetric chain decomposition of necklace posets}
\author{Vivek Dhand}

\begin{document}

\maketitle
\thispagestyle{empty}

\begin{abstract}
A finite ranked poset is called a symmetric chain order if it can be written as a disjoint union of rank-symmetric, saturated chains. If $\mathcal{P}$ is any symmetric chain order, we prove that $\mathcal{P}^n/\mathbb{Z}_n$ is also a symmetric chain order, where $\mathbb{Z}_n$ acts on $\mathcal{P}^n$ by cyclic permutation of the factors.
\end{abstract}

\se{Introduction}

Let $(\cP,<)$ be a finite poset.   A {\em chain} in $\cP$ is a sequence of the form $x_1 < x_2 < \dots < x_n $ where each $x_i \in \cP$.  For $x, y \in \cP$, we say $y$ {\em covers} $x$ (denoted $x \lessdot y$) if $x < y$ and there does not exist $z \in \cP$ such that $x < z$ and $z < y$.  A {\em saturated} chain in $\cP$ is a chain where each element is covered by the next. We say $\cP$ is {\em ranked} if there exists a function $\rk: \cP \to \bZ_{\geq 0}$ such that $x \lessdot y$ implies $\rk(y) = \rk(x) + 1$. The {\em rank} of $\cP$ is defined as $\rk(\cP) = \max\{\rk(x) \mid x \in \cP\} + \min\{\rk(x) \mid x \in \cP\}$.  A saturated chain $\{ x_1 \lessdot x_2 \lessdot \dots \lessdot x_n \}$ in a ranked poset $\cP$ is said to be {\em rank-symmetric} if $ \rk(x_1) + \rk(x_n) =   \rk(\cP)$. 

We say that $\cP$ has a {\em symmetric chain decomposition} if it can be written as a disjoint union of saturated, rank-symmetric chains.  A {\em symmetric chain order} is a finite ranked poset for which there exists a symmetric chain decomposition.

A finite product of symmetric chain orders is a symmetric chain order. This result can be proved by induction \cite{dB} or by explicit constructions (e.g.\! \cite{GK}).  Naturally, this raises the question of whether the quotient of a symmetric chain order under a given group action has a symmetric chain decomposition. For example, if $X$ is a set then $\bZ_n$ acts on the set $Map(\bZ_n,X) \simeq X^n$.  The elements of $X^n/\bZ_n$ are called $n$-{\em bead necklaces with labels in} $X$.  A symmetric chain decomposition of the poset of binary necklaces was first constructed by K. Jordan \cite{J}, building on the work of Griggs-Killian-Savage \cite{GKS}.  There have been recent independent proofs and generalizations of these results \cite{DMT,HS}.   The main result of this paper is the following:

\begin{thm}\tn{ If $\cP$ is a symmetric chain order, then $\cP^n/\bZ_n$ is a symmetric chain order. }\end{thm}

We give a brief outline of the proof.  First, we show that the poset of $n$-bead binary necklaces is isomorphic to the poset of {\em partition necklaces}, i.e.\ $n$-bead necklaces labeled by positive integers which sum to $n$.  It turns out to be convenient to exclude the maximal and minimal binary necklaces, which correspond to those partitions of $n$ having $n$ parts and $0$ parts, respectively.  Let $\cQ(n)$ denote the poset of partition necklaces with these two elements removed. We decompose $\cQ(n)$ into rank-symmetric sub-posets $\cQ_\al$, running over partition necklaces $\al$ where 1 does not appear.  This decomposition corresponds to the ``block-code" decomposition of binary necklaces defined in \cite{GKS}.  

We can also extend this idea to non-binary necklaces.  In fact, the poset of $n$-bead $(m{+}1)$-ary necklaces embeds into the poset of $nm$-bead binary necklaces, and the image corresponds to the union of those $\cQ_\al \su \cQ(mn)$ such that every part of $\al$ is divisible by $m$.  

Next, we prove a ``factorization property" for $\cQ_\al \su \cQ(n)$.  If $P$ and $Q$ are finite ranked posets, we say that $P$ {\em covers} $Q$ (or $Q$ is {\em covered by} $P$) if there is a morphism of ranked posets from $P$ to $Q$ which is a bijection on the underlying sets. We denote this relation as $P \cover Q$.  Note that any ranked poset covered by a symmetric chain order is also a symmetric chain order.   If $\al$ is aperiodic, then $\cQ_\al$ is covered by a product of symmetric chains.  If $\al$ is periodic of period $d$, then $\cQ_\al$ is covered by the poset of $(n/d)$-bead necklaces labeled by $\cQ_\be$, for some aperiodic $d$-bead necklace $\be$.  

Finally, if $\cP$ is a symmetric chain order, then $\cP^n/\bZ_n$ has a decomposition into posets which are either products of chains, or posets of $d$-bead necklaces with labels in a product of chains (where $d < n$), or posets of $n$-bead $(m{+}1)$-ary necklaces for some $m \geq 1$.  In each case, we apply induction to finish the proof.

\se{Generalities on necklaces}

We begin by recalling some basic facts about $\bZ_n$-actions on sets.  We will use additive notation for the group operation of $\bZ_n$. The subgroups of $\bZ_n$ are of the form $\lan d \ran$ where $d$ is a positive divisor of $n$, and $\bZ_n/\lan d \ran \simeq \bZ_d$. If $X$ is a set with $\bZ_n$-action, let $X^{\lan d \ran}$ denote the set of $\lan d\ran$-fixed points in $X$.  Equivalently:
\[ X^{\lan d \ran} = \{x \in X \mid \lan d\ran \su Stab_{\bZ_n}(x) \}.  \]
Note that $X^{\lan c \ran} \su X^{\lan d \ran}$ if $c$ is a divisor of $d$.  Next, we define:
\[ X^{\{d\}} = \{x \in X \mid \lan d\ran = Stab_{\bZ_n}(x) \}.  \]
Of course, we have:
\[ X = \bigsqcup_{d|n} X^{\{d\}} \]
and the $\bZ_n$ action on $X^{\{d\}}$ factors through $\bZ_d$.  In other words, we have a bijection:
\[ X/\bZ_n \simeq \bigsqcup_{d|n} X^{\{d\}}/\bZ_d.   \]

Now consider the special case where $X = Map(\bZ_n,Y)$ for some arbitrary set $Y$, where $\bZ_n$ acts on the first factor. In other words,
\[ (af)(b) = f(a+b) \]
for any $a,b \in \bZ_n$ and $f : \bZ_n \to Y$.    Now the previous paragraph implies that:
\[ Map(\bZ_n,Y) = \bigsqcup_{d|n} Map(\bZ_n,Y)^{\{d\}} \]  
and
\[ Map(\bZ_n,Y)/\bZ_n = \bigsqcup_{d|n} Map(\bZ_n,Y)^{\{d\}}/\bZ_d . \]  
The elements of $Map(\bZ_n,Y)/\bZ_n$ are called $n$-{\em bead necklaces with labels in} $Y$.  

An element of $Map(\bZ_n,Y)^{\{d\}}/\bZ_d$ is said to be {\em periodic of period} $d$.    An element of $Map(\bZ_n,Y)^{\{n\}}/\bZ_n$ is said to be {\em aperiodic}.   Given a map $g : \bZ_n \to Y$, let $[g]$ denote the corresponding necklace in $Map(\bZ_n,Y)/\bZ_n$.   A $n$-bead necklace with labels in $Y$ can be visualized as a sequence of $n$ elements of $Y$ placed evenly around a circle, where we discount the effect of rotation by any multiple of $\frac{2 \pi }{n}$ radians.  Given $(y_1, \dots, y_n) \in Y^n$, let $[y_1, \dots, y_n]$ denote the corresponding $n$-bead necklace.

Our first observation is that an $n$-bead necklace of period $d$ is uniquely determined by any sequence of $d$ consecutive elements around the circle.  Moreover, as we rotate the circle, these $d$ elements will behave exactly like an aperiodic $d$-bead necklace.

\begin{prop} \tn{ There is a natural bijection between $n$-bead necklaces of period $d$ and aperiodic $d$-bead necklaces. } \end{prop}

\begin{proof} Recall the following general fact: if $G$ is a group, $H$ is a normal subgroup of $G$, and $Y$ is an arbitrary set, then there is an isomorphism of $G$-sets:
\[ Map(G,Y)^H \simeq Map(G/H,Y) \]
\[ f\mapsto (gH \mapsto f(g) ). \]
Moreover, the action of $G$ on each side factors through $G/H$. 
In particular, there is an isomorphism of $\bZ_n$-sets:
\[ Map(\bZ_n,Y)^{\lan d \ran} \simeq Map(\bZ_d , Y)\]
where the $\bZ_n$-action factors through $\bZ_d$.  Looking at elements of period $d$, we get:
\[ Map(\bZ_n,Y)^{\{ d \}} \simeq Map(\bZ_d,Y)^{\{d\}} \]
and so:
\[ Map(\bZ_n,Y)^{\{ d \}}/\bZ_d  \simeq Map(\bZ_d,Y)^{\{d\}}/\bZ_d. \]
\end{proof}
Now suppose that $Y$ is a disjoint union of non-empty subsets: 
\[ Y = \bigsqcup_{i \in I} Y_i \]
where $I$ is a finite set. Equivalently, we have a surjective map $\pi: Y \to I $, where $Y_i = \pi \inv(i)$ for each $i \in I$.  It follows that there is a surjective map:
\[ \pi_*: Map(\bZ_n,Y) \to Map(\bZ_n,I) \]
\[ \pi_*(f) = \pi \circ f. \]
Given a map $g: \bZ_n \to I$, we define:
\[ Map_g(\bZ_n,Y) = \pi_*\inv (g) =\{ f: \bZ_n \to Y \mid \pi \circ f = g  \} . \]
In other words, $f \in Map_g(\bZ_n,Y)$ if and only if $f(a) \in Y_{g(a)}$ for all $a \in \bZ_n$. Since $\pi_*$ is surjective, we have a decomposition:
\[ Map(\bZ_n,Y) = \bigsqcup_{g \in Map(\bZ_n, I)}  Map_g (\bZ_n,Y).  \]
Note that $Map_g (\bZ_n,Y)$ is not necessarily stable under the action of $\bZ_n$. If $a,b \in \bZ_n$ and $f \in Map_g (\bZ_n,Y)$, then:
\[ a(f)(b) = f(a+b) \in Y_{g (a+b)}  \]
so we have a bijection:
\[ Map_g (\bZ_n,Y) \simeq Map_{a g}(\bZ_n,Y) \]
induced by the action of $a \in \bZ_n$.  We define:
\[ Map_{[g]}(\bZ_n,Y) = \bigcup_{a \in \bZ_n} Map_{a g}(\bZ_n,Y). \]
Note that $\bZ_n$ acts on $Map_{[g]}(\bZ_n,Y)$.

\begin{rmk}   \label{EqRel} \tn{ We recall a basic observation which will make it easier to define maps on sets of necklaces.  Suppose $S$ and $T$ are sets equipped with equivalence relations $\sim$ and $\approx$, respectively.  Let $U$ be a subset of $S$ which has a non-trivial intersection with each equivalence class in $S$.  Then $U$ inherits the equivalence relation $\sim$ and the natural map from $U/\!\!\sim$ to $S/\!\!\sim$ is a bijection. Given a map $f: U \to T$ such that $u_1 \sim u_2 \implies f(u_1) \approx f(u_2)$ for all $u_1, u_2 \in U$, we obtain a map $(S/\sim) \simeq (U/\sim) \to (T/\approx)$.}\end{rmk} 

\begin{rmk}  \tn{ If $\al$ is a periodic $n$-bead necklace of period $d$ with labels in $I$, then:
\[ \al = [\underbrace{\be, \dots, \be}_{\frac{n}{d} \tn{ times}}] \]
where $\be = (\be_1, \dots, \be_d)$ is a $d$-tuple of elements in $I$ such that $[\be]$ is aperiodic. } \end{rmk}

\begin{lem} \label{KeyLemma} \tn{ Let $\pi: Y \to I$ be a surjective map where $I$ is finite.}

\tn{(1)  There is a natural decomposition:
\[ Map(\bZ_n,Y)/\bZ_n = \bigsqcup_{d|n}  \left( \bigsqcup_{\al \in Map(\bZ_n,I)^{\{d\}}/\bZ_d} Map_\al (\bZ_n,Y)/\bZ_n \right)  . \]}

\tn{(2) If $\al = [\be, \dots, \be]  \in Map(\bZ_n,I)^{\{d\}}/\bZ_d$, where $\be = (\be_1, \dots, \be_d)$, then there is a bijection:
\[  Map_{\al}(\bZ_n,Y)/\bZ_n \simeq (Y_{\be_1} \times \dots \times Y_{\be_d})^{\frac{n}{d}}/\bZ_{\frac{n}{d}} . \]}

\end{lem} 

\begin{proof}
(1) Since 
\[ Map(\bZ_n,Y) = \bigsqcup_{g \in Map(\bZ_n, I)}  Map_g (\bZ_n,Y) \]
and
\[ Map(\bZ_n,I) =  \bigsqcup_{d|n} Map(\bZ_n,I)^{\{d\}} \]
we see that:
\[ Map(\bZ_n,Y) = \bigsqcup_{d|n} \left( \bigsqcup_{g \in Map(\bZ_n, I)^{\{d\}}} Map_g (\bZ_n,Y) \right).   \]
As noted above, in order to make this an equality of $\bZ_n$-sets we need to take the coarser decomposition: 
\[ Map(\bZ_n,Y) =  \bigsqcup_{d|n} \left( \bigsqcup_{[g] \in Map(\bZ_n, I)^{\{d\}}/\bZ_d} Map_{[g]}(\bZ_n,Y) \right). \]
Now we simply take the quotient by $\bZ_n$ on both sides:
\[ Map(\bZ_n,Y)/\bZ_n =  \bigsqcup_{d|n} \left( \bigsqcup_{[g] \in Map(\bZ_n, I)^{\{d\}}/\bZ_d} Map_{[g]}(\bZ_n,Y)/\bZ_n \right). \]
Note that we are simply organizing the $n$-bead $Y$-labeled necklaces by looking at the periods of the underlying $n$-bead $I$-labeled necklaces.

(2) Let $g \in Map(\bZ_n,I)^{\{d\}}$ and let $a\in \bZ_n$.  By definition, $ag = (a + x)g$ if and only if $x \in \lan d \ran$.  So:
\[ Map_{a g} (\bZ_n,Y) = Map_{(a+x)g}(\bZ_n,Y)  \]
if $x \in \lan d \ran$. On the other hand, if
\[ h \in Map_{a g} (\bZ_n,Y) \cap Map_{(a+x)g}(\bZ_n,Y) \]
for some $x \in \bZ_n$, then $\pi \circ h = a g = (a + x) g$, which implies that $x \in \lan d \ran$.  The upshot is that we can actually write $Map_{[g]}(\bZ_n,Y)$ as a {\em disjoint} union over $\bZ_d$:
\[ Map_{[g]}(\bZ_n,Y) = \bigsqcup_{a \in \bZ_d} Map_{a g}(\bZ_n,Y). \]
Now consider the sequence of values $g(a)$ for $a\in \bZ_n$.  This sequence is of the form $(\be, \dots, \be)$, where $\be = (\be_1, \dots, \be_d)$. Therefore:
\[ Map_{g} (\bZ_n,Y) \simeq (Y_{\be_1} \times \dots \times Y_{\be_d})^{\frac{n}{d}}\]
and so:
\[ Map_{[g]}(\bZ_n,Y)\simeq \bigsqcup_{j = 0}^{d-1} (Y_{\be_{j+1}} \times\dots\times Y_{\be_d} \times Y_{\be_1} \times \dots \times Y_{\be_{j}})^{\frac{n}{d}}.\]
Let us apply Remark \ref{EqRel} to the following sets: 
\[ S =   \bigsqcup_{j = 0}^{d-1} (Y_{\be_{j+1}} \times\dots\times Y_{\be_d} \times Y_{\be_1} \times \dots \times Y_{\be_{j}})^{\frac{n}{d}} \quad \tn{and} \quad T =  (Y_{\be_1} \times \dots \times Y_{\be_d})^{\frac{n}{d}}. \]
The equivalence relations on $S$ and $T$ are defined by group actions: $\bZ_n$ acts on $S \simeq Map_{[g]}(\bZ_n,Y)$ and $\bZ_{\frac{n}{d}}$ acts on $T$ by cyclic permutation of the factors. Let $U$ be the subset of $S$ corresponding to the $j = 0$ component:
\[ U = (Y_{\be_1} \times \dots \times Y_{\be_d})^{\frac{n}{d}}. \]
Each element of $S$ is equivalent to an element of $U$, and the restricted equivalence relation on $U$ is given by the action of the subgroup $\lan d \ran$ which is exactly the same as the action of $\bZ_{\frac{n}{d}}$ by cyclic permutation of the factors. Therefore:
\[ S/\bZ_n \simeq U/\lan d \ran \simeq T/\bZ_{\frac{n}{d}} . \]
\end{proof}

\begin{rmk} \tn{ We can visualize the above result as follows: we choose a place to ``cut" an $n$-bead $Y$-labeled necklace in order to get an $n$-tuple of elements of $Y$. We can always rotate the original necklace so that the underlying $I$-labeled necklace has a given position with respect to the cut. Moreover, if the underlying $I$-labeled necklace has period $d$, then we can break the $n$-tuple into segments of size $d$ so that the corresponding $I$-labeled $d$-bead necklaces are aperiodic.   As we rotate the original necklace by multiples of $\frac{2 \pi}{d}$ radians, we will permute these segments among each other.} \end{rmk}

\se{Partition necklaces}

Let $n$ be a positive integer. Consider the set of ordered partitions of $n$ into $r$ positive parts:
\[ \cP(n,r) = \{ (a_1, \dots, a_r) \in \bZ_{>0}^r \mid  \sum_{i = 1}^r a_i= n\}  \]
Define:
\[ \cP(n) = \bigsqcup_{r = 1}^{n-1} \cP(n,r) \]
In other words, $\cP(n)$ is the set of non-empty ordered partitions of $n$ into positive parts, where at least one part is greater than 1.  Note that refinement of partitions defines a partial order on $\cP(n)$, and the rank of a partition is given by the number of parts.  

Let $\cQ(n)$ denote the set of necklaces associated to $\cP(n)$:
\[ \cQ(n) = \bigsqcup_{i = 1}^{n-1} \cP(n,r)/\bZ_r \]
In other words:
\[ \cQ(n) = \{ [a_1, \dots, a_r] \in \bZ_{>0}^r/\bZ_r \mid  1 \leq r \leq n-1, \sum_{i = 1}^r a_i= n \} \]
where $[a_1, \dots, a_r]$ denotes the $\bZ_r$-orbit of $(a_1, \dots, a_r)$.

The elements of $\cQ(n)$ are called {\em partition necklaces}.  Note that $\cQ(n)$ inherits the structure of a ranked poset from $\cP(n)$. 

Let $\cN(n,1)$ denote the set of $n$-bead binary necklaces with the necklaces $[0,\dots,0]$ and $[1,\dots,1]$ removed.

\begin{prop}\tn{ \label{BinaryEncoding} For any $n \geq 1$, there is an isomorphism of ranked posets:
\[ \psi_n: \cN(n,1) \simeq \cQ(n) . \] }\end{prop}

\begin{proof}
Given a non-empty $n$-bead binary necklace $\be$ of rank $r$,  let $\psi_n(\be)$ be the necklace whose entries are given by the number of steps between consecutive non-zero entries of $\be$.  More precisely, $\psi_n$ is given by:
\[  [1,0^{c_1}, 1, 0^{c_2}, \dots, 1, 0^{c_r}] \mapsto [c_1 +1, \dots, c_r + 1] \]
Note that the right hand side is the necklace of a partition of $n$ into $r$ positive parts. The inverse of $\psi_n$ is given by:
\[ [a_1, \dots, a_r] \mapsto [1,0^{a_1-1}, 1, 0^{a_2-1}, \dots, 1, 0^{a_r - 1} ] . \] 
Moreover, changing a ``zero" to a ``one" in a binary necklace corresponds to a refinement of the corresponding partition necklace, so the above bijection is compatible with the partial orders and rank functions on each poset.
\end{proof}

An ordered partition $(a_1, \dots, a_r)$ and the corresponding partition necklace $[a_1, \dots, a_r]$ are said to be {\em fundamental} if each $a_i \geq 2$.  Let $\cF(n)$ denote the set of fundamental partition necklaces in $\cQ(n)$. 

Now we apply Remark \ref{EqRel} to the case where $S =\cP(n)$ and $T$ is the subset of $\cP(n)$ consisting of fundamental partitions.  Equip each set with the necklace equivalence relation, so $(S/\!\!\sim) = \cQ(n)$ and $(T/\!\!\approx) = \cF(n)$. Define the subset:
\[ U = \{  (1^{n_1},m_1,1^{n_2},m_2,...,1^{n_k},m_k) \in \cP(n) \mid n_i \geq 0, m_i \geq 2 \tn{ for all } 1 \leq i \leq k \} \]
Since we have excluded $(1,\dots,1)$ from $\cP(n)$, we see that any element of $\cP(n)$ is equivalent to some element in $U$.  Now define:
\[ f: U \to T \]
\[ (1^{n_1},m_1,1^{n_2},m_2,...,1^{n_k},m_k)\mapsto (m_1 + n_1, \dots,m_k + n_k). \]
Since $f$ is compatible with the respective equivalence relations, we obtain a map:
\[ \pi_n : \cQ(n) \to \cF(n) \]
\[ [1^{n_1}, m_1, 1^{n_2}, m_2, \dots, 1^{n_k}, m_k] \mapsto [m_1 + n_1, m_2 + n_2, \dots, m_k+n_k ] . \]
Note that $\pi_n$ restricts to the identity on $\cF(n)$.  In particular, $\pi_n$ is surjective. Therefore, we get a decomposition of $\cQ(n)$:
\[ \cQ(n) = \bigsqcup_{\al \in \cF(n)} \cQ_\al \]
where $\cQ_\al = \pi_n\inv(\al)$.   This decomposition is the same as the decomposition for binary necklaces defined in \cite{GKS}.  Indeed, the map $\pi_n \circ \psi_n$ is essentially the necklace version of the ``block-code" construction.  

If $m \geq 1$, a fundamental partition necklace $[a_1, \dots, a_r] \in \cF(n)$ is said to be {\em divisible} by $m$ if each $a_i$ is divisible by $m$.  Define the following sub-poset of $\cQ(n)$:
\[ \cQ(n,m) = \{\al \in \cQ(n) \mid \pi_n(\al) \tn{ is divisible by } m \} = \bigsqcup_{\underset{m|\al}{\al \in \cF(n)}} \cQ_\al . \]
Let $\cN(n,m)$ denote the set of $n$-bead $(m{+}1)$-ary necklaces with the necklaces $[0,\dots,0]$ and $[m,\dots,m]$ removed. We have the following generalization of Proposition 3.1.

\begin{lem}  \tn{ \label{GeneralEncoding} For any $n, m \geq 1$, there is an isomorphism of ranked posets:
\[ \psi_{n,m}: \cN(n,m) \simeq \cQ(mn,m). \] }\end{lem}

\begin{proof} 
Given an $n$-bead $(m{+}1)$-ary necklace, we construct an $mn$-bead binary necklace via the substitution: $j \mapsto 1^{j} 0^{m-j} $, and then we apply the map $\psi_{mn}$ from Proposition \ref{BinaryEncoding}.  This composition is clearly a morphism of ranked posets. Here is an explicit formula for $\psi_{n,m}$:
\[ [b_1, 0^{c_1}, b_2, 0^{c_2}, \dots, b_r, 0^{c_r}] \mapsto [1^{b_1-1}, m(c_1+1) - b_1 +1, \dots, 1^{b_r - 1}, m(c_r+1) - b_r + 1 ]  \]
where each $b_i\geq 1$ and $c_i\geq 0$. The sum of the terms in the partition necklace is:
\[  \sum_{i = 1}^r  (b_i -1 + m(c_i + 1) - b_i + 1) = m (r + \sum_{i = 1}^r c_i) = mn\]
as desired.  Let us check that $\pi_{mn} \circ \psi_{n,m} (\al)$ is divisible by $m$ for all $\al \in \cN(n,m)$.  Consider the element:
\[ \al = [b_1, 0^{c_1}, b_2, 0^{c_2}, \dots, b_r, 0^{c_r}] . \]   If $c_i > 0$ or $b_i < m$, then the terms $1^{b_i-1}$ and $m(c_i+1) - b_i +1$ in $\psi_{m,n}(\al)$ merge together under $\pi_{mn}$ to give $m(c_i+1)$.  On the other hand, whenever $b_i = m$ and $c_i = 0$, we will get a $1^m$ term in $\psi_{m,n}(\al)$.  Applying $\pi_{mn}$ will result in adding $m$ to the next occurrence of $m(c_j+1)$, where $c_j > 1$.  In other words:
\[ \pi_{mn}(\psi_{n,m}(\al)) = [me_1,\dots,me_s]\] 
where $\pi_n(c_1+1, \dots, c_r+1) = [e_1, \dots, e_s]$, and this result is indeed divisible by $m$.

By reversing the above process, we get a formula for the inverse of $\psi_{n,m}$.  An arbitrary element of $\cQ(mn,m)$ is of the form:
\[ [1^{n_1}, m_1, 1^{n_2}, m_2, \dots, 1^{n_k}, m_k]   \]
where each $m_i \geq 2$, each $m_i + n_i$ is divisible by $m$, and $\sum_{i=1}^k (m_i + n_i) = mn$.  The corresponding $mn$-bead binary necklace is:
\[ [1^{n_1+1}, 0^{m_1 - 1}, \dots , 1^{n_k+1}, 0^{m_k - 1}]. \]
Now we need to apply the substitution $ 1^{j} 0^{m-j}\mapsto j $.  Since $m_i + n_i$ is divisible by $m$, we can apply this to each block $(1^{n_i+1}, 0^{m_i - 1})$ separately.  Furthermore, we should break each block into segments of size $m$ and apply the substitution to each segment.  Therefore, $(1^{n_i+1}, 0^{m_i - 1})$ looks like:
\[ (\underbrace{1^m, 1^m, \dots, 1^m}_{q_i \tn{ times}}, 1^{r_i}, 0^{m-r_i}, 0^{m_i - 1 - (m-r_i)}). \]
where $q_i$ is the quotient of the division of $n_i + 1$ by $m$ and $r_i$ is the remainder. Note that $m_i - 1 - (m - r_i) = m_i - 1 -m + (n_i + 1 - mq_i) = m_i + n_i - mq_i -m$, which is divisible by $m$. Therefore, the inverse of $\psi_{n,m}$ is given by the following formula:
\[ [1^{n_1}, m_1, 1^{n_2}, m_2, \dots, 1^{n_k}, m_k]  \mapsto [m^{q_1},r_1, 0^{t_1} \dots, m^{q_k},r_k, 0^{t_k}] \]
where: 
\[ n_i+1 = mq_i + r_i \tn{ such that } 0 \leq r_i < m \]
and
\[ t_i = \frac{m_i + n_i}{m} - q_i - 1 .\]
Note that the number of beads in the above necklace is:
\[ \sum_{i = 1}^k  \left( q_i + 1 + \frac{m_i + n_i}{m} - q_i - 1\right) = \frac{1}{m} \sum_{i = 1}^k (m_i + n_i) = \frac{mn}{m} = n\]
as desired.
\end{proof}

\begin{lem} \label{Factorization}  \tn{ Let $\al = [a_1,\dots, a_r] \in \cF(n)$. If $\al$ is aperiodic, then:
\[ \cQ_{[a_1]} \times \dots \times \cQ_{[a_r]} \cover \cQ_\al. \]
If $\al$ is periodic of period $d$  and $\al = [\underbrace{\be, \dots, \be}_{\frac{r}{d} \tn{ times}}] $, then: 
\[\cQ_{[\be]}^{\frac{r}{d}}/\bZ_{\frac{r}{d}} \cover \cQ_\al.\] } 
\end{lem}

\begin{proof}
If $m \geq 2$, note that $\cQ_{[m]}$ is a chain with $m-1$ vertices.  We will apply Lemma \ref{KeyLemma} to the following set:
\[ \cQ = \bigsqcup_{m = 2}^n \cQ_{[m]}. \]
Note that our indexing set is $I = \{2, \dots, n\}$. Let $\al = [a_1, \dots,a_r] \in \cF(n)$. Since $a_1 + \dots + a_r = n$, we know that each $a_i \leq n$, which implies that $\al$ is labeled by elements of $I$.  If $\al$ is aperiodic, it follows from part (2) of Lemma \ref{KeyLemma} that we have a rank-preserving bijection:
\[ Map_\al(\bZ_r,\cQ) /\bZ_r \simeq \cQ_{[a_1]} \times \dots \times \cQ_{[a_r]} . \]
On the other hand, if $\al = [\be, \dots, \be]  \in Map(\bZ_r,I)^{\{d\}}/\bZ_d$, where $\be = (\be_1, \dots, \be_d)$, then we have rank-preserving bijections:
\[  Map_{\al}(\bZ_r,\cQ)/\bZ_r \simeq (\cQ_{[\be_1]} \times \dots \times \cQ_{[\be_d]})^{\frac{r}{d}}/\bZ_{\frac{r}{d}} \simeq \cQ_{[\be]}^{\frac{r}{d}}/\bZ_{\frac{r}{d}} \]
where the second bijection exists due to the fact that $[\be]$ is aperiodic.  It remains to check that the poset relations are preserved. Indeed, any covering relation among two necklaces labeled by $\cQ_{[\be_1]} \times \dots \times \cQ_{[\be_d]}$ will correspond to a covering relation within a chain $\cQ_{[\be_i]}$ for some $i$, which will also be a covering relation among the corresponding $\cQ$-labeled necklaces.
\end{proof}

\begin{rmk}\tn{ The above Lemma provides an explanation of why it is easier to find a symmetric chain decomposition of $n$-bead binary necklaces if $n$ in prime \cite{GKS}.  Indeed, in this case all non-trivial necklaces are aperiodic, so each $\cQ_\al$ is covered by a product of symmetric chains and we can apply the Greene-Kleitman rule. }\end{rmk}

\se{Proof of the theorem}

\begin{thm}\tn{ \label{MainTheorem} If $\cP$ is a symmetric chain order, then $\cP^n/\bZ_n$ is a symmetric chain order. }\end{thm}

\begin{proof}
The statement is trivial for $n = 1$. Assume that the theorem is true for any $n' < n$. Let $C_1, \dots, C_r$ denote the chains in a symmetric chain decomposition of $\cP$.  We may assume that:
\[ \cP = \bigsqcup_{i = 1}^r C_i . \]
If we let $I = \{1, 2, \dots, r\}$ and apply part (1) of Lemma \ref{KeyLemma} to $\cP$, we obtain:
\[ Map(\bZ_n,\cP)/\bZ_n = \bigsqcup_{d|n}  \left( \bigsqcup_{\al \in Map(\bZ_n,I)^{\{d\}}/\bZ_d} Map_\al (\bZ_n,\cP)/\bZ_n \right)  . \]
Now we apply part (2) of Lemma \ref{KeyLemma}. If $\al = [a_1, \dots, a_n]$ is an aperiodic $n$-bead necklace with labels in $I$, then:
\[  C_{a_1} \times \dots \times C_{a_n} \cover Map_{\al}(\bZ_n, \cP). \]
Since $C_{a_1} \times \dots \times C_{a_n}$ is a symmetric chain order, it follows that $Map_{\al}(\bZ_n, \cP)$ is a symmetric chain order. Also note that $C_{a_1} \times \dots \times C_{a_n}$ is a centered subposet of $Map(\bZ_n,\cP)/\bZ_n$. On the other hand, if $\al = [\be, \dots, \be]$ is a periodic $n$-bead necklace with labels in $I$, where $\be = (\be_1, \dots, \be_d)$, then:
\[ (C_{\be_1} \times \dots \times C_{\be_d})^\frac{n}{d}/\bZ_{\frac{n}{d}} \cover Map_{\al}(\bZ_n, \cP)/\bZ_n.\] 
Again, note that this poset is a centered subposet of $Map(\bZ_n,\cP)/\bZ_n$ since it is a cyclic quotient of a centered subposet of $\cP^n$.

If $d > 1$, then $\frac{n}{d} < n$ and $(C_{\be_1} \times \dots \times C_{\be_d})$ is a symmetric chain order, so 
\[(C_{\be_1} \times \dots \times C_{\be_d})^\frac{n}{d}/\bZ_{\frac{n}{d}} \] is a symmetric chain order by induction.

If $d = 1$, then:
\[ C^n/\bZ_n \cover Map_{\al}(\bZ_n, \cP)/\bZ_n   \]
where $C$ is a chain with $m+1$ vertices, for some $m \geq 1$.   It suffices to consider the centered subposet $\cN(n,m)$.  By Lemma \ref{GeneralEncoding}, we have: 
\[ \cN(n,m) \simeq \cQ(mn,m). \]
If $\cQ_\al \su \cQ(mn,m)$, then $\al = [ma_1, \dots, ma_s]$, where $a_1 + \dots + a_s = n$.  In particular, note that $s \leq n$. By Lemma \ref{Factorization}, there are two possibilities for $\cQ_\al$.  If $\al$ is aperiodic, $\cQ_\al$ is a product of chains, so it is a symmetric chain order.  If $\al$ is periodic of period $d$, then:
\[\cQ_{[\be]}^{\frac{s}{d}} / \bZ_{\frac{s}{d}}  \cover \cQ_\al  \]
where $[\be]$ is a $d$-bead aperiodic necklace.  In particular, $\cQ_{[\be]}$ is itself a product of chains (hence a symmetric chain order). We know that $\be = (mc_1, \dots, mc_d)$, where $c_1 + \dots + c_d = \frac{dn}{s}$.  There are three possible cases:

(i) If $d > 1$, then $\frac{s}{d} < n$.  Since $\cQ_{[\be]}$ is a symmetric chain order, by induction we conclude that
\[ \cQ_{[\be]}^{\frac{s}{d}} /\bZ_{\frac{s}{d}} \]
is a symmetric chain order.

(ii) If $d = 1$ and $s < n$ then $Q_{[\be]}$ is a single chain, so $Q_{[\be]}^s/\bZ_s$ is a symmetric chain order by induction. 

(iii) If $d = 1$ and $s = n$, then $\be = (m)$ and $\al = [m, \dots, m]$.  In this case:
\[\cQ_{[m]}^n/\bZ_{n} \cover \cQ_\al . \]
Since $\cQ_{[m]}$ is a chain with $m-1$ vertices, we see that we have returned to the case of the $\bZ_n$-quotient of the $n$-fold power of a single chain.  However, note that the we have managed to decrease the length of the chain by two, i.e. from $m+1$ vertices to $m-1$ vertices.  Now we can again apply Lemma \ref{GeneralEncoding} and Lemma \ref{Factorization} to the centered subposet  $\cN(n,m-2)$, etc.

Eventually, after we go through this argument enough times, we will eventually reach the case of:
\[ C^n/\bZ_n \]
where $C$ is a chain with one or two vertices.  If $|C| = 1$, there is nothing to show.  So we are left with the case where $C$ is a chain with two vertices, i.e. the poset of binary necklaces.  It suffices to look at the centered subposet $\cN(n,1)$.  By Proposition \ref{BinaryEncoding}, 
\[ \cN(n,1) \simeq \cQ(n). \]
Again, we consider the subposets $\cQ_\al$. As usual, if $\al$ is aperiodic then $\cQ_\al$ is covered by a product of symmetric chains. If $\al = [\be, \dots, \be]$ is periodic of period $d$ then
\[ \cQ_{[\be]}^{\frac{n}{d}}/\bZ_{\frac{n}{d}} \cover \cQ_{\al}  \]
where $[\be]$ is an aperiodic $d$-bead necklace and $\cQ_{[\be]}$ is a product of chains.  If $d > 1$, then $\frac{n}{d} < n$ so
\[ \cQ_{[\be]}^{\frac{n}{d}}/\bZ_{\frac{n}{d}} \]
is a symmetric chain order by induction.  Finally, if $\al$ is periodic of period $d = 1$ then $\al$ is an $n$-bead partition necklace of period 1 whose entries sum to $n$, so $\al = [1,1, \dots,1]$, but this element was explicitly excluded from the set $\cQ(n)$.
\end{proof}

{\bf Acknowledgements.}  I would like to thank the Department of Mathematics at Michigan State University for their hospitality.  I am especially grateful to Bruce Sagan for his encouragement while this project was under way.  This paper also benefited greatly from several referee comments.

\end{document}